\documentclass[a4paper, 12pt]{article}
\usepackage[left=3cm,right=2cm,top=2.5cm,bottom=2.5cm]{geometry}

\usepackage[utf8]{inputenc}
\usepackage{lmodern,stackrel,verbatim,textcomp}
\usepackage[T1]{fontenc}
\usepackage[pdftex]{color,graphicx}
\usepackage{upgreek,graphicx, epstopdf,dsfont,soul}
\usepackage{mathtools,listings,marginnote}
\usepackage{amsmath,amsfonts,amssymb,amsthm,mathrsfs,stackrel}
\usepackage{xargs} 
\usepackage{tikz}
\usetikzlibrary{shapes,arrows,calc,positioning}
\usepackage{pgfplots}
\usepackage{xcolor}
\usepackage{caption}
\usepackage{subcaption}

\usepackage[normalem]{ulem}

\pgfplotsset{compat=1.17}

\usepackage[nottoc]{tocbibind}
\usepackage{hyperref}
\hypersetup{
    colorlinks=true,
    linkcolor=blue,
    filecolor=blue,
    urlcolor=blue,
    }

\renewcommand{\leq}{\leqslant}
\renewcommand{\geq}{\geqslant}
\renewcommand{\log}[1]{\mathrm{ln}\left(#1\right)}
\newcommand{\ind}{\mathbf{1}}
\renewcommand{\P}[1]{\mathbb{P}\left(#1\right)}
\newcommand{\e}{\varepsilon}
\newcommand{\Z}{{\mathbb{Z}}}
\newcommand{\N}{{\mathbb{N}}}
\newcommand{\R}{\mathbb{R}}
\newcommand{\G}[1]{\mathrm{Geo}\left(#1\right)}

\newcommand{\Ex}[1]{\mathbb{E}\left(#1\right)}
\newcommand{\Var}[1]{\mathrm{Var}\left(#1\right)}

\newcommand{\li}{\mathrm{li}}
\newcommand{\lii}{\left[\mathrm{li}^{-1}\right]}

\newcommand{\SV}{\mathrm{SV}_0}
\theoremstyle{plain}
\newtheorem{Thm}{Theorem}

\newtheorem{Prop}{Proposition}
\newtheorem{Cor}{Corollary}
\newtheorem{Lem}{Lemma}

\newtheorem{Conj}{Conjecture}
\theoremstyle{definition}

\newtheorem{Rmk}{Remark}

\begin{document}
\title{First order of the renewal covering of the natural numbers}
\author{Alberto M. Campos}
\maketitle

\begin{abstract}
This paper introduces a new type of covering process that covers the set of natural numbers using renewal processes as objects. Inspired by the behavior of prime numbers, the model in each step finds the smallest vacant point, $k$, and place, starting in $k$, a renewal process with a step distribution given by a geometric random variable with parameter $\frac{1}{k}$. The model depends on its entire past, and small perturbations in its initial value can lead to very different outcomes. Here, we expose a technique that finds the first-order limit behavior for the number of objects placed until $n$, which exhibits intriguing similarities to prime number distributions, having a concentration around $n\log{n}$.
\end{abstract}

\section{Introduction}\label{sec:Intro}\noindent

This paper studies a collection of dependent renewal processes that fully covers the natural numbers, denoted by $\N=\{1,2,\cdots\}$. The problem itself belongs to a class of probability problems related to covering processes, where given a space, in this case $\N$, the process performs a cover of $\N$ using random objects that for us will be random renewal sets. In function of that, denote this process as a \textbf{renewal covering of the natural numbers}. The problem treated in this paper resembles the structure and behavior of prime numbers; this similarity is delicate and will be addressed clearly after the statements of the theorem. However, when defining the process, it is interesting for the reader to keep in mind a parallel to the set of primes.

Given a set of i.i.d. random variable $(G_k)_k$, define the \textbf{renewal set} associated with the sequence $(G_k)_k$ as the set $\{\sum_{k=1}^{n} G_k\}_n\subseteq \N$. This paper always considers $G_k$ to be distributed as a \textbf{geometric random variable} with some parameter $p\in[0,1]$, denoted by $\G{p}$, where $\P{\G{p}=k}=(1-p)^{k-1}p$ for $k\in \{1,2,\cdots\}$.

Before giving a formal definition of the model, let us describe the process in words. The process consists of a sequence $(P_n)_n$ constructed through induction.
\begin{itemize}
    \item Start by choosing $P_1=2$ almost surely.
    \item In the set of points greater than $P_1$, i.e., $\{3,4,\cdots\}$, sample for each point $k$ an independent Bernoulli $\xi_k^1$ with parameter $\frac{1}{P_1}$ ($=\frac{1}{2}$). Whenever, $\{\xi_k^1=1\}$ occurs, we say that $k$ is a random multiple of $P_1$. 
    \item Let $\mathcal{M}_1=\{k\in \{3,4,\cdots\}: \xi_k^1=1\}$ be the set of random multiples of $P_1$, then take the smallest natural non-multiple of $P_1$, i.e. $\min\left\{\{3,4,\cdots\}\setminus \mathcal{M}_1\right\}$, to be the next element of the sequence $P_2$.  
\end{itemize} 
Notice that, with probability one, the set $\mathcal{M}_1$ is different from $\{3,4,\cdots\}$, so one can construct the element $P_2$ almost surely. Moreover, one have that $P_2-P_1\sim \G{\left(1-\frac{1}{2}\right)}$. 

With $n\geq 1$ assume that the elements $P_1$,$\cdots$ and $P_{n}$, with its sets of multiples $\mathcal{M}_1$,$\cdots$ and $\mathcal{M}_{n}$. To find the next pair of elements $P_{n+1}$ and $\mathcal{M}_{n+1}$, do as follows: 
\begin{itemize}
    \item The element $P_{n+1}$ is going to be the smallest natural greater than $P_n$, that is not multiple of $P_1$, $\cdots$ and $P_n$, i.e., $\min\left\{\{P_{n}+1,P_{n}+2,\cdots\}\setminus\bigcup_{k=1}^n \mathcal{M}_k\right\}$. 
    \item With $P_{n+1}$ fixed, sample for each natural $k$ bigger than $P_{n+1}$, an independent Bernoulli random variable $\xi_k^{n+1}$; Then, set $\mathcal{M}_{n+1}=\{k\in \{P_{n+1}+1,P_{n+1}+2,\cdots\}: \, \xi_k^{n+1}=1\}$. 
\end{itemize}
Equally to the first step in the construction, notice that $P_{n+1}$ exists almost surely. And, by calculation, we get $P_j-P_{j-1}\sim \G{\prod_{k=1}^j\left(1-\frac{1}{P_j}\right)}$.\bigskip

The construction above resembles the structure of prime numbers, where the next prime is the smallest number that is not a multiple of any previous prime. However, the problem proposed here is random and depends on all its past to be described, so the values and behavior of the process can be very different from those known from algebra. Let us go ahead and say that with high probability, the value of the random variable $P_n$ is different from the value of the $n-th$ prime. 

Locally, the process may be different from the prime numbers, however, it will be demonstrated that the total number of objects needed to cover the interval $\{2,3,\cdots,n\}$ corresponds in first order to the number of primes up to the value $n$, that is, the prime counting function $\pi(n)$. Since $P_n$ is not a prime number, to avoid a possible misleading notation, denote $P_n$ as the $n-th$ \textbf{generator}; one can think of it as a random-ideal generator in the context of algebra. 

To describe the process rigorously, consider the space of increasing sequence $\Omega=\{(q_n)_n: q_n\in \N,\, q_n>q_{n-1}\}$. Let $\mathcal{F}_n$ be the $\sigma-$algebra generated by the cylinder of size $n$, that is, $\{(q_1,\cdots, q_k)\in \N^{k}: 1\leq q_1<q_2<\cdots <q_k,\, k\leq n\}$, and set $\mathcal{F}=\sigma\left(\bigcup_n\mathcal{F}_n\right)$ the \textbf{cylinder $\sigma-$algebra}. 

In such a measurable space, define inductively the renewal covering of the natural numbers. For this, define a set of \textbf{ generators} $(P_n)_n$ inductively. Where $P_1=2$ almost surely, and for each $n>1$ set: 
\begin{align}\label{eq:DefiniçãoPn}
P_{n+1}^{}&=P_{n}^{}+ \G{\prod_{k=1}^{n}\left(1-\frac{1}{P_k^{}}\right) }.
\end{align} Where, for any $j\geq 1$, given a increasing sequence of naturals $1\leq a_1<\cdots<a_n$,  the geometric satisfies:
\begin{align}\label{eq:independentgeogiven}
    &\P{\G{\prod_{k=1}^{n}\left(1-\frac{1}{P_k^{}}\right)}=j\middle|P_1^{}=a_1,\cdots,P_n^{}=a_{n}}\nonumber\\  &\quad \,\quad\, = \P{\G{\prod_{k=1}^{n}\left(1-\frac{1}{a_k}\right)}=j}. 
\end{align} In words, this means that the parameter of the geometric depends in the value of the first $n$ elements, i.e. belongs to $\mathcal{F}_n=\sigma(P_1,\cdots,P_n)$, but fixed the values of $(P_1,\cdots,P_n)$ the value of the geometric is independent. To simplify the notation in the proof, it is worth to define the random parameter:
\begin{align}\label{eq:Definiçãolambdan}
    \lambda_n= \prod_{k=1}^{n}\left(1-\frac{1}{P_k}\right).
\end{align} 

Our result consists of finding the first-order term of the sequence $(P_n)_n$, thus doing a law of large numbers for such a sequence. 
\begin{Thm}\label{thm:1}
    Consider $(P_n)_n$ a renewal covering of the natural numbers defined above. Then: 
    \begin{align*}
        \frac{P_n}{n\log{n}}\to 1 \quad \text{ almost surely}.
    \end{align*}
\end{Thm}

At first glance, the model and the result do not seem surprising, since they consist of a law of large numbers for the sum of geometric random variables. However, the process depends on all its past, $P_n\in \mathcal{F}_{n-1}=\sigma(P_1,\cdots,P_{n-1})$, and the model is a non-Markovian process; small perturbations in its initial value can lead to very different outcomes. To the best of our knowledge, the way the paper deals with the dependencies of the process is a novelty. 
\bigskip

This idea of creating random prime numbers is not new in the literature. Landau in 1903, \cite{Landau1903}, showed using local analytic lemmas how to derive information about the zeta function using the set of primes. Then, in $1937$, Beurling proposed an inverse problem, the Beurling generalized prime systems \cite{Beurling1937}, where fixing a sequence $(P_n)_n$, and controlling the size of the multiplicative semigroup generated by it, is possible to control how dense such a set is. The idea behind Beurling's construction was to create conditions under which if the prime numbers were satisfied, the Riemann hypothesis would be true. 
This line of study is very active and has several applications and contributions; see Diamond (2005) \cite{Diamond2005} and Olofsson (2011) \cite{Olofsson2011}.  

Another set of techniques, associated with random primes, is to define a random prime gap with some distribution, usually $\G{\ln^{-1}\left(n\right)}$. The analyzes of such sequence of random gaps can provide insight into the behavior of the prime numbers; see Crammer \cite{Cramr1936}, where in $1936$ he conjectured that the prime gap is of order $\ln^2\left(n\right)$. 

This paper shifts the focus of the deterministic zeta function to a random perspective that tries to see this algebraic structure as some fixed point in the sense of limiting distribution. The goal is not to create a random model that copies the prime numbers but to see if in this random process of filling gaps there exists a deterministic law of large numbers. It is left as an open problem to determine all the first-order terms for processes where, for some $\alpha>0$ the generators behave as $P_{n+1}=P_{n}+\G{\prod_{k=1}^n \left(1-P_k^{-\alpha}\right)}$. 

Another difference between this article and the previous literature is the notation and the way the results are expressed. Generally, number theory and renewal theory work with the inverse function, while in this text, due to the techniques developed, we chose to work with the sequence $ (P_n)_n $. Despite this difference, it is simple to show that, with the concentration of $P_n$, inverting it, the number of generators until $n$ concentrates around $\frac{n}{\log{n}}$. 

To clarify some connection's, due to Dursatt (1999) \cite{Dusart1999} and Rosser (1941) \cite{Rosser1941}, the $n-th$ prime $p_n$, for $n\geq 7$ satisfies:
\begin{align}\label{eq:DursattRosserbound}
    n\log{n}+n\ln{\log{n}}-n\leq p_n\leq n\log{n}+ n\ln{\log{n}}. 
\end{align} Our proof shows that the renewal covering have the first order term in the limit behaving equal to the prime number. We believe that the second order term is also present in the problem, but  the same is not true for the third-order term, leaving open the following conjecture: 
\begin{Conj}\label{Conj:1}
    Consider $(P_n)_n$ a renewal covering of the natural numbers defined above. Then, there exists a random variable $Z$, such that:  
    \begin{align*}
        \frac{P_n-n\log{n}-n\ln{\log{n}}}{n}\overset{D}{\implies}Z, 
    \end{align*} where the convergence is in distribution. 
\end{Conj}

\textbf{Organization.} The structure of the paper is organized as follows: Section \ref{sec:Sketch} provides an sketch of the proof. Sections \ref{sec:Domination} and \ref{sec:Tightness} discuss the behavior of the random variable $\lambda_n$. In Section \ref{sec:Concentration}, a type of concentration result for the sequence $P_n$ is exposed. Section \ref{sec:law1} translates these fluctuations into the law of large numbers, thereby proving Theorem \ref{thm:1}.

\textbf{Notations.}The set of integers is denoted by $\Z=\{\cdots,-1,0,1,\cdots\}$,  the set of real numbers is going to be $\R$. 
The symbol $\forall$ means "for all" and the symbol  $\exists$ means "exists".

\textbf{Acknowledgment.} The author would like to thank Marco Aymone for his insightful discussions and valuable information regarding the literature on random primes.

\section{Proof Sketch of the proof}\label{sec:Sketch}\noindent

The proof of Theorem \ref{thm:1} consists of a long argument with three main steps. This outline will briefly explain each step, providing the general goals of the main result. To motivate the reader, the explanation given here follows a natural order (not the order of the proof). The best way to explain the proof is to start with the objective of the second step, then show a brief argument on how to finish the proof of the theorem, and finally explain the initial step.

The main objective of the proof is to smooth out the dependence that exists in the sequence $(P_n)_n$. For this, in the second step of the proof, we expose a concentration result for the sequence $(P_n)_n$ that resembles what is expected for the sums of independent geometric random variables. Specifically, remember the definition of the parameter $\lambda_n=\prod_{k=1}^n\left(1-\frac{1}{P_k}\right)$ and the fact that $P_n$ is a sum of geometric random variables with parameter $\lambda_n$. Then, for every $\e>0$ and $\alpha>1$, there exists $n_0=n_0(\e,\alpha)$, such that with high probability:
\begin{align}\label{eq:Psk-concentration}
    \P{\left|P_n-\sum_{k=1}^n \frac{1}{\lambda_k}\right|<\sqrt{n^{\alpha }} ,\,\forall n>n_0}>1-\e. 
\end{align}

At first glance, this probability gives the concentration of the problem. However, notice that $\lambda_n$ is also an unknown random variable. The reader can think that result \eqref{eq:Psk-concentration} is as a pseudo-concentration, that shows that two random variables, $\mathcal{F}_n$ measurable, are close together. 

The second object in this proof is subtle, but has a nice explanation. Basically, it consists of transforming the pseudo concentration \eqref{eq:Psk-concentration} into a law of large numbers. To briefly explain how, start by recalling that any random variable $\mathcal{F}_n$ measurable is written as a function $F$ of the values of $P_1$, $\cdots$, $P_n$. Therefore, equation \eqref{eq:Psk-concentration} has two possibilities: The first case is when the difference between the random variables is trivially true, and the second case is when the difference is not trivial and thus not every sequence $(P_n)_n$ will satisfy the equation. For example, the function $F(x_1,\cdots,x_n)=x_n$ implies that $P_n-F(P_1,\cdots, P_n)=0$ for every $n$, but the function $F(x_1,\cdots,x_n)=x_1+\cdots+x_n$ implies that $P_n$ is close to its summation; by some approximations, it is possible to relate the growth of $P_n$ with an exponential growth (the only type in which its value is close to its integral). 

Curiously, in this paper, the inequality leads to an integral inequality that can be related to a differential equation, and another field appears. The solutions of the differential equation is the inverse function of the logarithmic integral function, believed to be the right order of $\pi(n)$ see \cite{Riemann2008}. In particular, the solution have a growth of order $n\log{n}+n\log{\log{n}}$, and by some comparative arguments one can conclude that the values of $P_n$ are close to the solution of this differential equation. 

The first step in proof, and what will imply in the pseudo-concentration, is a tightness problem. As mentioned, when conditioning in the set of parameters $\lambda_n$, the quenched problem satisfies the equation \eqref{eq:Psk-concentration}. To go from quenched space to annealed space, the sequence $\lambda_n$ has to belong with high probability to a specific set of sequences. More precisely,  $\lambda_n$ is going to be tight in the space of slowly varying sequences, the sequences $(a_n)_n$ such that for every $t>0$, $\frac{a_{nt}}{a_n}$ converges to one when $n$ grows. 

\section{Stochastic Domination}\label{sec:Domination}\noindent

Let $X$ and $Y$ be two random variables, we say that $X$ stochastic dominate $Y$ if $\P{X\geq x}\geq \P{Y\geq x}$ for every $x\in \R$, whenever this happens, denote it as $X\succeq_{est} Y$  or $Y\preceq_{est} X$. Specifically for geometric random variables, let $1\geq p>q \geq 0$, when $X\sim \G{q}$ and $Y\sim \G{p}$, then $X\succeq_{est} Y$. 

Stochastic domination is not concerned with the independence of random variables. Therefore, just the fact that $X_k\succeq_{est} Y_k$ for every $k\geq 0$ is not enough to conclude that $\sum_{k=1}^n X_k\succeq_{est} \sum_{k= 1}^{n} Y_k$. However, considering the condition that the sequences $(X_k)_k$ and $(Y_k)_k$ are independent, one can use the stochastic domination in each term to dominate the sum. The renewal coverage o $\N$ does not satisfy the independence condition, but due to the hypothesis imposed by the equation \eqref{eq:independentgeogiven} that condition on the values of $P_n$ the geometric random variable is independent, it is possible to find similar results given specific events.

\begin{Lem} \label{Lem:estocdomsequence}Let integers  $n\geq 0$ and $m>0$, and set any $p\in [0,1]$. Define $(G_k)_k$ as an i.i.d. sequence, where $G_k\sim \G{p}$. Then:
\begin{align*}
   \ind\{\lambda_m \leq p\}\sum_{k=m}^{m+n} \G{\lambda_k} &\succeq_{est} \ind\{\lambda_m \leq p\}\sum_{k=m}^{m+n} G_k .
\end{align*}
\end{Lem}
\begin{proof} The idea of the proof is to observe that $(\lambda_n)_n$ is a monotonous sequence, then if at some point the parameter become smaller than $p$, it will stay smaller than $p$. Therefore, after this moment the parameter is smaller than $p$, and the increment of the problem at each step stochastic dominates an independent geometric with parameter $p$. 

To be rigorous, the proof uses induction. Start by fixing $m>0$ and $p\in [0,1]$. In the initial case, when $n=0$, define $\Gamma_m^0=\{(a_1,\cdots,a_m)\in \N^m: \prod_{k=1}^m \left(1-\frac{1}{a_k}\right)\leq p,\, 2\leq a_1<a_2<\cdots<a_m\}$, and for any $x\in \R$, the problem satisfies the following: 
\begin{align*}
    \mathbb{P}\Big(\lambda_m \leq p \,&;\, \G{\lambda_{m}}\geq  \,x\Big)\\&= \sum_{\vec{a}\in \Gamma_m^0}\P{\G{\lambda_{m}}\geq x\,\middle| (P_1,\cdots,P_m)=\vec{a}} \P{ (P_1,\cdots,P_m)=\vec{a}}.
\end{align*} By condition \eqref{eq:independentgeogiven} and independence of the random variables, since $\G{q}\succeq_{est} \G{p}$ for any values of $1\geq p>q \geq 0$, one get that:
\begin{align*}
    \P{\lambda_m \leq p\, ;\, \G{\lambda_{m}}\geq  x}&= \sum_{\vec{a}\in \Gamma_m^0}\P{\G{\prod_{k=1}^m\left(1-\frac{1}{a_k}\right)}\geq x} \P{ (P_1,\cdots,P_m)=\vec{a}}\\
    &\geq \P{\lambda_m\leq p}\P{G_m\geq x}= \P{\lambda_m\leq p\, ;\, G_m\geq x}. 
\end{align*}

Now, assume for any fixed $n-1\geq 0$ that:
\begin{align}\label{eq:inductioninstocasticdomination}
    \P{\lambda_m \leq p\, ;\,\sum_{k=m}^{m+n-1} \G{\lambda_k}\geq x} \geq \P{\lambda_m \leq p\, ;\,\sum_{k=m}^{m+n-1} G_k\geq x},\,  \forall x\in \R.
\end{align}

We are going to prove the same inequality for $n$.  Let $\Gamma^{n}_m=\{(a_1,\cdots,a_{m+n})\in \N^{m+n}: \prod_{k=1}^m \left(1-\frac{1}{a_k}\right)\leq p,\, 2\leq a_1<a_2<\cdots<a_{m+n}\}$. Then, using condition \eqref{eq:independentgeogiven}, for any $x\in \R$, it is true that: 
\begin{align*}
    &\P{\lambda_m \leq p\, ;\,\sum_{k=m}^{m+n} \G{\lambda_k}\geq x}\\
    &=\sum_{\vec{a}\in \Gamma_m^{n}}\P{\G{\prod_{k=1}^{m+n}\left(1-\frac{1}{P_k}\right)}+P_{m+n-1}\geq x+P_{m-1}\, ;\, (P_1,\cdots,P_{m+n})=\vec{a}}\\
    &=\sum_{\vec{a}\in \Gamma_m^{n+1}}\P{\G{\prod_{k=1}^{m+n}\left(1-\frac{1}{a_k}\right)}+a_{m+n-1}\geq x+a_{m-1}}\P{(P_1,\cdots,P_{m+n})=\vec{a}}\\
    &\geq \sum_{\vec{a}\in \Gamma_m^{n+1}}\P{\G{p}+a_{m+n-1}\geq x+a_{m-1}}\P{(P_1,\cdots,P_{m+n})=\vec{a}}\\
    &= \P{\lambda_m \leq p\, ;\, G_{m+n}+\sum_{k=m}^{m+n-1} \G{\lambda_k}\geq x}.
\end{align*} Now, since $G_{m+n}$ is independent, conditioning in its value and using the induction hypotheses of equation \eqref{eq:inductioninstocasticdomination} , one get that: 
\begin{align*}
    \mathbb{P}\Bigg(\lambda_m \leq p\, &;\, G_{m+n}+\sum_{k=m}^{m+n-1} \G{\lambda_k}\geq x\Bigg)=\\
    &=\sum_{\ell=1}^{\infty}  \P{\lambda_m \leq p\, ;\, \sum_{k=m}^{m+n-1} \G{\lambda_k}\geq x-\ell}\P{G_{n+m}=\ell}\\
    &\geq  \sum_{\ell=1}^{\infty}  \P{\lambda_m \leq p\, ;\,\sum_{k=m}^{m+n-1} G_k\geq x-\ell}\P{G_{n+m}=\ell}\\
    &= \P{\lambda_m \leq p\, ;\,\sum_{k=m}^{m+n} G_k\geq x}.
\end{align*} As desired. 
\end{proof}

Analogously to the Lemma \ref{Lem:estocdomsequence}, it is possible to get the other side of the stochastic domination.
\begin{Lem}\label{Lem:estocdomsequence2}
    Let $n>0$ be an integer and set $p\in [0,1]$. Define $(G_k)_k$ as an i.i.d. sequence, where $G_k\sim \G{p}$. Then:
    \begin{align*}
        \ind\{\lambda_k \geq p, \forall k\geq 0 \} \sum_{k=1}^{n} \G{\lambda_k} &\preceq_{est}  \ind\{\lambda_k \geq p, \forall k\geq 0 \} \sum_{k=1}^{n} G_k.
    \end{align*}
\end{Lem} The proof is analogous and can be done following the same steps proved in Lemma \ref{Lem:estocdomsequence}, so it is left as an exercise. 

To make full use of the stochastic domination with respect to the sum of independent geometric random variables, consider the classical result of large deviation theory; see Section 2.7 of Durrett \cite{Durrett2019} for a proof. 
\begin{Prop}\label{Prop:LargeDeviationGeoindpendent}
    Let $p\in[0,1]$, and $(G_k)_k$ and i.i.d. sequence with $G_k\sim \G{p}$. Then, for every $\e>0$, there exists a value $\mathcal{I}(\e)>0$ such that: \begin{align*}
        \P{\left|\sum_{k=1}^n G_k-\frac{n}{p}\right|>\e n}\leq \exp\{-n\mathcal{I}(\e)\}.
    \end{align*}
\end{Prop}
\bigskip

With all the initial results and definitions of the stochastic domination, it is possible to give the first two lemmas about the sequence $(P_n)_n$ and $(\lambda_n)_n$. More specifically, it will be shown that almost surely $\lambda_n$ converges to zero and that $(P_n)_n$ grows faster than any linear function.

\begin{Lem}\label{Lem:lambdatozero}
    The sequence $(\lambda_n)_n$ converges to zero with probability one.  
\end{Lem}
\begin{proof}
 For any $n>0$, the value of $\lambda_n$ is a product of numbers less than one; therefore, it is a monotone bounded decreasing sequence and must converge to some number with probability one. Now,  assuming by absurd that it do not converge to zero almost surely, then there exists an $c\in(0,1)$ such that:
 \begin{align}\label{eq:lemalimitemaiorquezeroc}
     \P{ \lambda_n\geq c,\, \forall n>0}>0.
 \end{align} To finish, we are going to show that there exists an event $A$, such that $\{\{\lambda_n\geq c,\,\forall n>0\}\, ;\,A\}$ have the same probability of the events $\{\{\lambda_n\geq c,\,\forall n>0\}\, ;\,A^c\}$ and $\{ \lambda_n\geq c,\, \forall n>0\}$. Implying that equation \eqref{eq:lemalimitemaiorquezeroc} is false, concluding this proof. 

Notice that if $\{\lambda_n\geq c,\, \forall n>0\}$ occur, using the formula of $\lambda_n$ in \eqref{eq:Definiçãolambdan}, and that $1-x\leq e^{-x}$ for all $x>0$, one gets that:
 \begin{align*}
     c\leq \lambda_n= \prod_{k=1}^n\left(1-\frac{1}{P_k}\right)\leq \exp\{-\sum_{k=1}^n \frac{1}{P_k}\}.
 \end{align*} And thus, the summation $\sum_{k=1}^{n} \frac{1}{P_k}$ must converge. Then:
 \begin{align*}
     \left\{\lambda_n\geq c,\, \forall n>0\right\}&=\left\{\lambda_n\geq c, \,\forall n>0 \, ;\, \sum_{k=1}^{\infty} \frac{1}{P_k}<\infty\right\}, \text{ and }\\
      \P{\lambda_n\geq c,\, \forall n>0}&=\P{\lambda_n\geq c, \,\forall n>0 \, ;\, \sum_{k=1}^{\infty} \frac{1}{P_k}<\infty}.
 \end{align*} 

 On the other hand, assuming that $\{\lambda_n\geq c, \, \forall n>0\}$ occurs, by Lemma \ref{Lem:estocdomsequence2}, it is possible to stochastically dominate the random variable $P_n$.  Let $(G_k)_k$ and i.i.d. sequence where $G_k\sim \G{c}$, then: 
   \begin{align}\label{eq:stocasticdominationcporcima}
    \ind\{\lambda_n \geq c, \forall n\geq 0 \} \sum_{k=1}^{n} \G{\lambda_k} &\preceq_{est}  \ind\{\lambda_n \geq c, \forall n\geq 0 \} \sum_{k=1}^{n} G_k.
    \end{align} Since $P_n=2+\sum_{k=1}^{n} \G{\lambda_k}$, it is possible to control its value using the independent sum of the geometric random variable in the equation \eqref{eq:stocasticdominationcporcima}. Using the large deviation of Proposition \ref{Prop:LargeDeviationGeoindpendent} together with Borel-Cantelli, it is possible to conclude that the event $\{\exists n_0\text{ s.t. }P_n<\frac{2n}{c},  \, \forall n>n_0 \}$ occurs with probability one. But, if $P_n<\frac{2n}{c}$ for every $n>n_0$, then $\sum_{k=1}^{\infty} \frac{1}{P_k}$ diverges. Thus:
\begin{align*}
    \P{\lambda_n\geq c,\, \forall n}=\P{\lambda_n\geq c, \forall n\, ;\,\sum_{k=1}^\infty \frac{1}{P_k}=\infty}.
\end{align*} Concluding that for every $c>0$ the probability of the event $\{ \lambda_n\geq c,\, \forall n>0\}$ must be  zero, and with probability one, one get that $(\lambda_n)_n$ converges to zero. 
\end{proof}

\begin{Lem}\label{lem:Pnmaiorquelinear}
    For every $C>0$, the sequence $(P_n)_n$ satisfies the following: \begin{align*}
        \P{\exists n_0 \text{ s.t. } P_n>Cn, \,  \forall n>n_0}=1.
    \end{align*}
\end{Lem}
\begin{proof}
     This proof uses a similar approach of the stochastic domination in Lemma \ref{Lem:lambdatozero} where it is proven that $\lambda_n$ converges to zero almost surely. Here, it is going to be used the stochastic domination of Lemma \ref{Lem:estocdomsequence}, where  the random variables $P_n$ will eventually dominates the sum of i.i.d. geometric random variables with arbitrary small parameters, implying that $P_n> Cn$ for every $C>0$ and $n$ large enough.

     Notice that $(\lambda_n)_n$ is a monotonous decreasing sequence with probability one. Moreover, by Lemma \ref{Lem:lambdatozero}, the sequence $(\lambda_n)_n$ converges almost surely to zero, implying that for every $C>0$ and every small $\e>0$, there exists a $n_1=n_1(\e,C)$ such that for every $n>n_1$:
    \begin{align}\label{eq:lambdanlessc1e}
        \P{\lambda_n\leq \frac{1}{2C}}\geq 1-\e.
    \end{align}  

    Furthermore, when the parameter becomes smaller than $\frac{1}{2C}$, the increment of the sequence $P_n$ will dominate a geometric random variable with a parameter $\frac{1}{2C}$. Thus, letting $(G_k)_k$ an i.i.d. sequence where $G_k\sim \G{\frac{1}{2C}}$, using Lemma \ref{Lem:estocdomsequence}, one gets that:
    \begin{align}\label{eq:stocasticdominationintheproof}
        \ind\left\{\lambda_{n_1}\leq \frac{1}{2C}\right\} P_{n_1+n}\succeq_{est}\ind\left\{\lambda_{n_1}\leq \frac{1}{2C}\right\} \sum_{k=1}^n G_k. 
    \end{align} 

    To finalize the proof, one can use Proposition \ref{Prop:LargeDeviationGeoindpendent} (large deviation principle) and Borel Cantelli lemma, to get that the event $\{\exists n_0 \text{ s.t. }$ $\sum_{k=1}^n  \G{\frac{1}{2C}}>Cn, \, \forall n>n_0\}$ happens with probability one. Since $n_1=n_1(\e,C)$ is finite, using the probability of equation \eqref{eq:lambdanlessc1e}, one get: 
    \begin{align*}
        \P{\lambda_{n_1}\leq \frac{1}{2C}\,;\,\left\{ \exists n_0 \text{ s.t. } \sum_{k=1}^n  \G{\frac{1}{2C}}>Cn, \,\forall n>n_0\right\}} =\P{\lambda_{n_1}\leq \frac{1}{2C}}\geq 1-\e.
    \end{align*} Which concludes by stochastic domination in equation \eqref{eq:stocasticdominationintheproof}, that for every $\e>0$: 
    \begin{align*}
        \P{\exists n_0 \text{ s.t. } P_n>Cn, \,\forall n>n_0}\geq 1-\e. 
    \end{align*}And, since $\e>0$ is arbitrary, we conclude the proof. 
\end{proof}

\section{Tightness in the slowly varying}\label{sec:Tightness}\noindent

Our goal now is to find the fluctuations of the random variable $\lambda_n$ when $n$ is large. For this, we will show that the sequence $\lambda_n$ belongs to a class of sequences that have a well behavior at infinity with probability one. Specifically, the set of slowly varying functions; see Chapter 0 from \cite{Resnick1987}. Moreover, to be self-contained in this subsection, we adapted the theory of slowly varying functions and prove the theorems needed for our result. 

For any number $x\in \R$, and sequence $(\gamma_n)_n$, define $\gamma(x)=\gamma_{\lceil x\rceil}$. Then, we say that $\gamma$ is \textbf{slowly varying at infinity} if for every $t>0$, one get that:
\begin{align*}
    \lim_{x\to \infty} \frac{\gamma(xt)}{\gamma(x)}=1.
\end{align*}Moreover, denote by $\SV$ the \textbf{set of all slowly varying sequences at infinity}. 

\begin{Lem}\label{Lem:lambdaslowv}
 The random sequence $(\lambda_n)_n$ belongs to $\SV$ with probability one. More than this, for every $\e>0$, $\delta>0$ and $t>0$, there exists $x_0=x_0(\e,\delta,t)$ such that: 
\begin{align}\label{eq:quantitativelambdaslow}
    \P{1-\delta\leq \frac{\lambda(xt)}{\lambda(x)}\leq 1+\delta, \, \forall x>x_0}>1-\e. 
\end{align}
\end{Lem}

The quantitative equation \ref{eq:quantitativelambdaslow} is reefer as a\textbf{ tightness }result. To be precise, note that $0<\lambda(x)\leq \frac{1}{2}$, and $(\lambda_n)_n$ is a monotonous decreasing sequence. Then, Lemma \ref{Lem:lambdaslowv} shows that the random function $\frac{\lambda(xt)}{\lambda(x)}$ concentrates its probability mass in the compact $\prod_{k=1}^{\lfloor x_0\rfloor}[0,2] \times \prod_{k=\lceil x_0\rceil }^{\infty} [1-\delta,1+\delta]$ (infinite product of compacts). 

\begin{proof}
    This proof uses the Lemma \ref{lem:Pnmaiorquelinear} which proves $P_n$  grows faster than any linear function, to show that the ratio $\frac{\lambda(xt)}{\lambda(x)}$ is close to one with high probability.
    
    For this, fix $t\geq1$, then by the definition \eqref{eq:Definiçãolambdan}:  
    \begin{align*}
    \frac{\lambda(xt)}{\lambda(x)}&=\frac{\prod\limits_{k=1}^{\lfloor xt\rfloor} \left(1-\frac{1}{P_k}\right) }{\prod\limits_{j=1}^{\lfloor x\rfloor} \left(1-\frac{1}{P_j}\right) }=\prod_{k=\lceil x \rceil}^{\lfloor xt\rfloor}\left(1-\frac{1}{P_k}\right).
    \end{align*} Since $P_n$ is a monotonous sequence and bigger than $2$ almost surely, using that $1-x\leq e^{-x}$, one can concludes with $t\geq 1$ that:
    \begin{align}
      \left(1-\frac{1}{P_{\lceil x\rceil}}\right)^{x(t-1)}&\leq\frac{\lambda(xt)}{\lambda(x)}\leq 1,\text{ that implies }\nonumber\\
     \exp\left\{- \frac{x(t-1)}{2P_{\lceil x\rceil}}\right\}&\leq\frac{\lambda(xt)}{\lambda(x)}\leq  1. \label{eq:boundintherateoflambdaslow}
    \end{align}
    
    Using Lemma \ref{lem:Pnmaiorquelinear}, for every $C>0$, exists $n_1=n_1(C,\e)$ such that:
    \begin{align*}
        \P{P_n>Cn,\,\forall n>n_1}= \sum_{k=0}^{n_1} \P{P_k<Ck, \, \{P_n>Cn,\,\forall n>k\}}>1-\e. 
    \end{align*}

    In particular, fixing $t\geq 1$ and any  $\delta>0$, there exists $C=C(\delta,t)$ accordingly so that
    \begin{align*}
        \exp\left\{- \frac{x(t-1)}{2Cx}\right\}>1-\delta,
    \end{align*} now taking $x_0=x_0(\e,\delta,t)=\lceil n_1\rceil$, one get deterministically by equation \eqref{eq:boundintherateoflambdaslow} that if $\{P_n>Cn,\,\forall n>n_1\}$ occurs, then the ratio $\frac{\lambda(xt)}{\lambda(x)}$ belongs to $(1-\delta,1]$, implying that:
    \begin{align*}
        \P{1-\delta \leq \frac{\lambda(xt)}{\lambda(x)}\leq 1,\,\forall x>x_0}\geq  \P{P_n>Cn,  \,\forall n>n_1}\geq 1-\e.
    \end{align*}    
    
    Now, when $t\in(0,1)$, the problem satisfy:
    \begin{align}\label{eq:elavaivoltar}
        1\leq \frac{\lambda(xt)}{\lambda(x)}\leq \left(1-\frac{1}{P_{\lceil x\rceil}}\right)^{-x(1-t)}\leq \exp\left\{\frac{x(1-t)}{P_{\lceil x\rceil }}\right\}
    \end{align} Then, by choosing another value of $C=C(\delta,t)$ accordingly, one gets that there exists $x_0^*=x_0^*(\e,\delta,t)$ such that that: 
    \begin{align*}
        \P{1 \leq \frac{\lambda(xt)}{\lambda(x)}\leq 1+\delta,\,\forall x>x_0^*}\geq 1-\e.
    \end{align*}
    
    Moreover, by equation \eqref{eq:quantitativelambdaslow}, since $\delta$ and $\e$ are arbitrary, $(\lambda_n)_n$ is monotonous, then one gets $\lambda(x)\in \SV$ almost surely. Concluding the proof.
\end{proof}

The above result is also true for some sequences related to $\lambda_n$. Specifically:
\begin{Cor}\label{Cor:lambdaslowv}
     The random  sequences $(\lambda_n^{-1})_n$ and $(\lambda_n)^{-2}$ belongs to $\SV$ with probability one. Moreover, for every $\e>0$, $\delta>0$ and $t>0$, there exists $x_0=x_0(\e,\delta,t)$ such that: 
\begin{align*}
    \P{1-\delta\leq \frac{\lambda^{-1}(xt)}{\lambda^{-1}(x)}\leq 1+\delta, \, \forall x>x_0}>1-\e,\\ 
     \P{1-\delta\leq \frac{\lambda^{-2}(xt)}{\lambda^{-2}(x)}\leq 1+\delta, \, \forall x>x_0}>1-\e.
\end{align*}
\end{Cor}The proof follows the same ideas of Lemma \ref{Lem:lambdaslowv}, just noticing that if $\sqrt{1-\delta}\leq \frac{\lambda^{-1}(xt)}{\lambda^{-1}(x)}\leq \sqrt{1+\delta}$, then $1-\delta\leq \frac{\lambda^{-2}(xt)}{\lambda^{-2}(x)}\leq 1+\delta$. So, we left this proof as an exercise. 
\bigskip

For slowly varying sequences $(\gamma_n)_n$, it is possible to control the sum of the values $\sum_{k=1}^{n} \frac{1}{\gamma_n}$ with some precision. The theorem that propose such relation is called Karamatha's theorem. In essence, by controlling the fraction $\frac{\gamma^{-1}(xt)}{\gamma(x)}$, it is possible to control the summation of these values, see Theorem 0.6 in S. I. Resnick \cite{Resnick1987}. 

In this paper, it is not possible to directly apply Karamatha's theorem, since the sequence $\lambda_n$ is a distribution over $\SV$, and not just a fixed sequence. To solve this, it is worth to observe the proof of such theorems with some attention. Therefore, a simplified version of the original theorem is proved.

\begin{Prop}[Karamatha]\label{Prop:Karamathalimit1} Consider any deterministic function $\gamma(x)\in \SV$, such that $\gamma(x)> 0$ for all $x>0$, then:

    \textbf{A-} For every $\alpha>0$, one get that $\lim\limits_{x\to \infty} x^{-\alpha}\gamma(x)=0.$
    
    \textbf{B-} If $\gamma(x)$ is a monotonous increasing function, one get that:
    \begin{align*}
        \lim\limits_{x\to \infty}\frac{\sum\limits_{k=1}^x \gamma(k)}{x\gamma(x)}=1.
    \end{align*}
\end{Prop}
\begin{proof}
     To prove affirmation \textbf{A} with fixed $\gamma\in \SV$, take any $t>1$, $\alpha>0$, and find $\delta>0$ that satisfies $1+\delta<t^{\alpha}$.  Using the definition of slowly varying functions, there exists $x_1=x_1(\delta,t)$ such that for every $x>x_1$:
    \begin{align}\label{eq:lambdaugrowsnotpoly}
        \frac{\gamma(xt)}{\gamma(x)}\leq 1+\delta.
    \end{align} This implies that for any $x>x_1$, there exist $x^*=x^*(x,x_1)$ and an integer $n\geq 0$, such that $x=x^*t^n$ for some $x^*\in(x_1,tx_1]$; Thus:
    \begin{align}
        0\leq \frac{\gamma(x)}{x^{\alpha}}=\frac{\gamma(x^*t^n)}{(x^*t)^{\alpha n}}&=\frac{1}{(x^*t)^{\alpha n}}  \frac{\gamma(x^*t^n)}{\gamma(x^*t^{n-1})}\cdots \frac{\gamma(x^*t)}{\gamma(x^*)} \gamma(x^*)\nonumber\\\label{eq:boundsslowlyvaryingpolyequationsup}
        &\leq \frac{\gamma(x^*)}{{x^*}^{\alpha}}\left(\frac{1+\delta}{t^{\alpha}}\right)^n\leq \left(\frac{1+\delta}{t^{\alpha}}\right)^n\sup_{y\in(x_1,tx_1]} \frac{\gamma(y)}{y^{\alpha}}.
     \end{align} Since $1+\delta<t^{\alpha}$, one can conclude that $\frac{1+\delta}{t^{\alpha}}<1$. Then, taking $x$ growing to $\infty$, since the supreme is a finite number, one gets that $x^{-\alpha}\gamma(x)$ converges to zero. Finishing the proof of \textbf{A}.

    The proof \textbf{B} consists in finding an upper and lower bound for the summation $\sum\limits_{k=1}^x \gamma(k)$. Fixing any integer $\ell>1$ and $\delta>0$, using the fact that $\gamma(x)\in \SV$, there exists $x_0(\delta,\ell)$ such that for every $x>x_0$:
    \begin{align}\label{eq:lambdaboundstoimplykara}
        1-\delta\leq  \frac{\gamma\left(x\frac{j}{\ell}\right)}{\gamma(x)}\leq  1+\delta, \quad \forall j\in\{1,\cdots,\ell\}.
    \end{align} Therefore, by the monotonicity of $\gamma(x)$, it is possible to conclude that: 
    \begin{align*}
        \frac{\sum\limits_{k=1}^x \gamma(k)}{x\gamma(x)}&= \sum_{j=0}^{\ell-1} \frac{1}{x\gamma(x)}\sum_{k=\frac{j}{\ell}x+ 1}^{\frac{j+1}{\ell}x} \gamma(k)  \leq \sum_{j=0}^{\ell -1} \frac{\frac{x}{\ell}\gamma \left(x\frac{j+1}{\ell}\right)}{x\gamma(x)}\leq 1+\delta.
    \end{align*} For the other side, note that:
    \begin{align*}
        \frac{\sum\limits_{k=1}^x \gamma(k)}{x\gamma(x)}\geq \frac{\sum\limits_{k=1}^{\frac{x}{\ell}} \gamma(k)}{x\gamma(x)}+ \frac{\ell-1}{\ell}(1-\delta) \geq \frac{\ell-1}{\ell}(1-\delta).  
    \end{align*} In particular, for any $\ell>0$ fixed and $\delta>0$, there exists a $x_0(\delta,\ell)$ such that for every $x>x_0$: 
    \begin{align}\label{eq:lowerboundwithkaramataslow}
         (1-\delta)- \frac{(1-\delta)}{\ell}\leq \frac{\sum_{k=1}^x\gamma(k)}{x\gamma(x)}&\leq 1+\delta.  
    \end{align} The proof of affirmation \textbf{B} follows by choosing the value of $\delta$ and $\ell$ accordingly. 
\end{proof}

\begin{Rmk}
    The theorem proposed in Proposition \ref{Prop:Karamathalimit1} is simpler than the actual Karamatha's theorem, for more detail see Theorem 0.6 of \cite{Resnick1987}.
\end{Rmk}

Now, back to the random sequence $\lambda(x)$, it is possible to approximate Karamatha's theorems with high probability. Specifically:

\begin{Prop}\label{Prop:highprobkaramatharate}
    For all $\e>0$ and $\delta>0$, there exists $x_0=x_0(\e,\delta)$ such that:
    \begin{align}\label{eq:highprobkaramatharate}
        \P{1-\delta \leq\frac{\sum_{k=1}^x \frac{1}{\lambda(k)}}{\frac{x}{\lambda(x)}}\leq 1+\delta, \, \forall x>x_0}>1-\e.
    \end{align} Moreover, for every $\alpha>0$ and $\eta>0$, there exists $x_1=x_1(\e,\eta,\alpha)$ such that: 
    \begin{align}\label{eq:highprobdonotgrowfasterthanpoly}
         \P{x^{-\alpha} \frac{1}{\lambda(x)}\leq \eta,\, \forall x>x_1}>1-\e.
    \end{align}
\end{Prop}   
\begin{proof}
    Using Corollary \ref{Cor:lambdaslowv}, for any $\e>0$, $\delta>0$, and $\ell>0$, there exists $x_0=x_0(\e,\delta,\ell)$ such that
    \begin{align}\label{eq:primeirolimirfracolambda}
        \P{1-\delta \leq  \frac{\lambda^{-1}\left(x\frac{j}{\ell}\right)}{\lambda^{-1}(x)}\leq  1+\delta, \forall x>x_0,\, \forall j\in\{1,\cdots,\ell\}}>1-\e.
    \end{align} Then, by Karamathas's Proposition \ref{Prop:Karamathalimit1}, more precisely by equation \eqref{eq:lowerboundwithkaramataslow}, it is true that: 
    \begin{align*}
          \P{1-\delta- \frac{(1-\delta)}{\ell} \leq\frac{\sum_{k=1}^x \frac{1}{\lambda(k)}}{\frac{x}{\lambda(x)}}\leq 1+\delta, \, \forall x>x_0}>1-\e.
    \end{align*} Again, by choosing $\delta$ and $\ell$ accordingly, one can conclude the equation \eqref{eq:highprobkaramatharate} as desired. 

    Now, to show equation \eqref{eq:highprobdonotgrowfasterthanpoly}, with fixed $\alpha>0$, $\eta>0$ and $\e>0$, fix $t=2$, and let $\delta=\delta(\alpha)<2^{\alpha}-1$. In particular, one get by Corollary \ref{Cor:lambdaslowv}, that there exists $x_1'=x_1'(\e,\delta(\alpha))$ such that: 
    \begin{align*}
        \P{\frac{\frac{1}{\lambda(2x)}}{\frac{1}{\lambda\left(x\right)}}\leq  1+\delta, \forall x>x_1'}>1-\frac{\e}{2}.
    \end{align*} Therefore, by equation \eqref{eq:boundsslowlyvaryingpolyequationsup} of Proposition \ref{Prop:Karamathalimit1}, it is possible to concludes that: 
    \begin{align*}
        \P{\frac{1}{x^{\alpha}\lambda(x)}\leq \eta, \forall x>x_1'} \geq \P{\left(\frac{1+\delta}{2^{\alpha}}\right)^n\sup_{y\in(x_1',2x_1']} \frac{1}{y^{\alpha}\lambda(y)}<\eta}.
    \end{align*}

    Finally, since $\lambda$ is monotonous, notice that $\sup_{y\in(x_1',2x_1']} \frac{1}{y^{\alpha}\lambda(y)}< \left(\lambda(2x_1')(x_1')^{\alpha}\right)^{-1}$ is a random variable with a finite mean. So, using the Markov inequality, it is possible to find $M=M(x_1',\alpha)$ such that:
    \begin{align*}
        \P{\sup_{y\in(x_1',2x_1']} \frac{\lambda(y)}{y^{\alpha}}<M}>1-\frac{\e}{2}.
    \end{align*} In particular, with $M>0$ fixed, one gets that exists $x_1=x_1(\e,\eta,\alpha)$ such that: 
    \begin{align*}
        \P{\frac{\lambda(x)}{x^{\alpha}}\leq \eta, \forall x>x_1}>1-\e. 
    \end{align*} As desired. 
\end{proof}

Moreover, using analogous arguments, it is possible to proof the same result for the sequence $\lambda^{-2}(x)$, that is:
\begin{Cor}\label{Cor:highprobkaramatharate}
    For all $\e>0$ and $\delta>0$, there exists $x_0=x_0(\e,\delta)$ such that:
    \begin{align*}
        \P{1-\delta \leq\frac{\sum_{k=1}^x \frac{1}{\lambda^2(k)}}{\frac{x}{\lambda^2(x)}}\leq 1+\delta, \, \forall x>x_0}>1-\e.
    \end{align*} Moreover, for every $\alpha>0$ and $\eta>0$, there exists $x_1=x_1(\e,\eta,\alpha)$ such that: 
    \begin{align*}
         \P{x^{-\alpha} \frac{1}{\lambda^2(x)}\leq \eta,\, \forall x>x_1}>1-\e.
    \end{align*}
\end{Cor}  The proof of this Corollary is analogous to the proof of Proposition \ref{Prop:highprobkaramatharate}, and it is left as an exercise. 

\section{Concentration results}\label{sec:Concentration}\noindent 

This subsection will demonstrate two different types of concentration for the random sequence $P_n$. Both results show that $P_n$ is concentrated around the random variable $\sum_{k=1}^n \lambda^{-1}_k$. They follow the same structure: In the quenched space, fixing any slowly varying sequence of parameters, it is possible to apply classical concentration results in the sum of geometric random variable; then using the tightness of the sequence $(\lambda_n)_n$ in the slowly varying space, it is possible to obtain such a concentration in the annealed space.  Specifically, the first result in Lemma \ref{Lem:firstconcentration} relies on Chebychev's inequality, and the second result in Proposition \ref{Prop:secondconcentrationgood} is proven using a version of McDiarmid's inequality, which makes it more precise.

\begin{Lem}\label{Lem:firstconcentration}
    Consider the dependent random variables $P_n$ and $\sum_{k=1}^n \frac{1}{\lambda_k}$, where $\lambda_k$ is equal to $\prod_{j=1}^k\left(1-\frac{1}{P_j}\right)$. For every $\alpha>1$, it is true that:
    \begin{align*}
        \lim_{n\to \infty}\P{\left|P_n-\sum_{k=1}^n \frac{1}{\lambda_k}\right|>\sqrt{n^{\alpha}}}=0.
    \end{align*}
\end{Lem}

\begin{proof}
    Concentration results for the random variable $P_n$ are very difficult to find. The main issue is that a small perturbation in the initial values can lead to very different outcomes due to the dependencies of $P_n$. However, it is possible to study the limiting behavior of this random variable by knowing what is the overall possible behavior of the parameters; then using quenched and annealed techniques, it is possible to get the desired result. 

    Consider $\Gamma=\{(\gamma)_n: \gamma_n\downarrow 0, \gamma(x)\in \SV\}$,  the set of all positive decreasing sequences that also belong to the set of slowly varying sequences. Then, for any fixed sequence $\gamma\in \Gamma$, define:
    \begin{align}
        S_n(\gamma)=\sum_{k=1}^{n} \G{\gamma_n},
    \end{align}  where the geometric random variables are independent. In particular, by some computations:
    \begin{align*}
        \Ex{S_n(\gamma)}= \sum_{k=1}^{n} \frac{1}{\gamma_k},\text{ and }
        \Var{S_n(\gamma)}=\sum_{k=1}^n \frac{1-\gamma_k}{\gamma_k^2}.
    \end{align*} Since $(\gamma_n)_n$ is slowly varying, the sequences $\left(\frac{1}{\gamma_n}\right)_n$ and $\left(\frac{1}{\gamma_n^2}\right)_n$ are also slowly varying. This means, using Proposition \ref{Prop:Karamathalimit1}, that:
    \begin{align} \label{eq:limdeterministicomeanv}
        \lim_{n\to \infty}\frac{\Ex{S_n(\gamma)}}{n\frac{1}{\gamma_n}}=1, \text{ and }
        \lim_{n\to \infty}\frac{\Var{S_n(\gamma)}}{n\frac{1}{\gamma_n^2}}=1.
    \end{align} So, for every $\alpha>1$, using Chebyshev's inequality, Proposition \ref{Prop:Karamathalimit1} and the limit of equation \eqref{eq:limdeterministicomeanv}, one get that:
    \begin{align}\label{eq:boundpolyinthelimitgamma}
        \lim_{n\to \infty}\P{\left|S_n(\gamma)-\sum_{k=1}^n \frac{1}{\gamma_k}\right|>\sqrt{n^{\alpha}}}\leq \lim_{n\to \infty} \frac{\Var{S_n(\gamma)}}{n\frac{1}{\gamma_n^2}} \frac{1}{n^{\alpha-1} \gamma^2_n}= 0. 
    \end{align} 

    To apply these results in the distribution of $\lambda_n$, consider $(\gamma_k)_{k=1}^n$ the first $n$ terms that belong to the sequence $\gamma\in \Gamma$, set $d\lambda_n$ a measure over $[0,1]^n$ defined as: 
    \begin{align*}
        \P{\lambda_1=\gamma_1,\cdots, \lambda_n=\gamma_n}= \int_{[0,1]^n} \ind\{\vec{x}=(\gamma_1,\cdots, \gamma_n)\} d\lambda_n(\vec{x}).
    \end{align*} In particular, one have that: 
    \begin{align}\label{eq:Integralrelationquencedsd}
        \P{\left|P_n-\sum_{k=1}^{n} \frac{1}{\lambda_k}\right|>\sqrt{n^{\alpha}}}= \int_{[0,1]^n}\P{\left|S_n(\gamma)-\sum_{k=1}^{n} \frac{1}{\gamma_k}\right|>\sqrt{n^{\alpha}}}  d\lambda((\gamma_k)_{k=1}^n).
    \end{align}

    It is proved that for each fixed $\gamma$, the probability in question goes to zero. This does not solve the problem, since the limit might not be uniform in $n$, and $\lambda_n$ can give mass to sequences that is not possible to be concentrated at time $n$. But this is not the case. Proposition \ref{Prop:highprobkaramatharate} proves a tightness condition for the random variable $\lambda_n$, that for every $\e>0$, $\delta>0$, $\eta>0$ and $\alpha>1$, one get that exists $n_0=n_0(\e,\delta,\eta,\alpha)$ such that: 
    \begin{align}\label{eq:t1}
        \P{1+\delta \geq\frac{\sum_{k=1}^n \frac{1}{\lambda(k)}}{\frac{n}{\lambda_n}}\geq 1-\delta, \, \forall n>n_0}>1-\frac{\e}{3}.\\
        \P{1+\delta \geq\frac{\sum_{k=1}^n \frac{1}{\lambda^2(k)}}{\frac{n}{\lambda_n^2}}\geq 1-\delta, \, \forall n>n_0}>1-\frac{\e}{3}.\label{eq:t2}\\
        \P{n^{1-\alpha} \frac{1}{\lambda_n^2}\leq \eta,\, \forall n>n_0}>1-\frac{\e}{3}.  \label{eq:t3}  
    \end{align} 
    
    All the events have high probability, then the intersection of all three events is going to have a probability greater then $1-\e$. Moreover,  by fixing any sequence $(\gamma_n)_n$ that satisfies all the events, one obtains for every $n>n_0$, using Chebyshev's inequality:
    \begin{align}\label{eq:boundinthequenchedwordprob}
        \P{\left|S_n(\gamma)-\sum_{k=1}^{n} \frac{1}{\gamma_k}\right|>\sqrt{n^{\alpha}}}\leq \frac{(1+\delta)n \frac{1}{\gamma_n^2}}{n\frac{1}{\gamma_n^2}} \frac{1}{n^{\alpha-1} \gamma_n^2} \leq \eta(1+\delta). 
    \end{align}

    Therefore, in the set $\Gamma$ of slowly varying sequences, there exist two disjoint sets, the set formed by sequences that satisfies all tightness events in equations \eqref{eq:t1},\eqref{eq:t2} and \eqref{eq:t3}; and the set of sequences that do not. In the set of sequences that satisfies all the events, the probability of $\P{\left|S_n(\gamma)-\sum_{k=1}^{n} \frac{1}{\gamma_k}\right|>\sqrt{n^{\alpha}}}$ is bounded by $\eta(1+\delta)$. In sequences that do not satisfy one of the events, the measure $d\lambda_n$ gives a mass less than $\e>0$.  Thus, by equation \eqref{eq:Integralrelationquencedsd}, for every $n>n_0$, one get:  
    \begin{align*}
        \P{\left|P_n-\sum_{k=1}^{n} \frac{1}{\lambda_k}\right|>\sqrt{n^{\alpha}}}\leq (1+\delta)\eta +\e. 
    \end{align*} Since $\e$, $\delta$ and $\eta$ are arbitrary values, taking all close to zero and choosing $n_0$ large enough. The concentration is proved as desired. 
\end{proof}

The result of Lemma \ref{Lem:firstconcentration} is not strong enough. We can not have any information about the distribution of $P_n$ and the distribution of the sum $\sum_{k=1}^n \frac{1}{\lambda_k}$ if they eventually come close together. To make a sharper result, consider the following version of McDiarmid's inequality \cite{McDiarmid1989} for Geometric random variables. This proof will significantly improve the polynomial decay of equation \ref{eq:boundpolyinthelimitgamma}, making it summable. 

\begin{Prop}
    Take $(\gamma_n)_n$ any monotonous decreasing sequence such that $\gamma(x)\in \SV$. Define the independent sum of geometric random variables with parameters $\gamma_k$ as $S_n(\gamma)=\sum_{k=1}^{n} \G{\gamma_k}$. Then, for every $\e>0$, $\alpha>1$, there exists $n_0=n_0(\e,\alpha)$ such that: 
    \begin{align*}
        \P{\left|S_n(\gamma)-\sum_{k=1}^n \frac{1}{\gamma_k}\right|>\sqrt{n^{\alpha}},\,\forall n>n_0} \leq \e.
    \end{align*}
\end{Prop}

\begin{proof}
    The proof consists in using McDiarmid's inequality for the sum of independent geometric random variables, but notice that the geometric random variable is unbounded. To apply McDiarmird's inequality, fix $\e>0$, then let $\beta=\beta(\e)>1$ be such that $\frac{1}{2^{\beta}}+\frac{1}{(1-\beta)2^{\beta-1}}<\frac{\e}{2}$. Define the truncated random variable: 
    \begin{align*}
    G_k=\G{\gamma_k} \ind\left\{\G{\gamma_k} \leq \beta\gamma_k^{-1} \ln{(k)} \right\}.    
    \end{align*}

    About the truncated random variable $G_k$, the probability that it is different from $\G{\gamma_k}$   is equal to:
    \begin{align}\label{eq:Gkneqgeosomavel}
        \P{G_k\neq \G{\gamma_k}}=(1-\gamma_k)^{\beta \gamma_k^{-1}\ln{(k)}}\leq \frac{1}{k^\beta}.
    \end{align} Therefore, we can conclude that: \begin{align}\label{eq:Geometricdifferenceboundfinite}
        \P{G_k=\G{\gamma_k},\,\forall k\geq 2}\geq 1- \sum_{k=2}^{\infty} \frac{1}{k^\beta}\geq 1-\frac{1}{2^{\beta}}-\frac{1}{(1-\beta)2^{\beta-1}}>1-\frac{\e}{2}.
    \end{align}

    Now, apply McDiarmid's inequality in the truncate random variable $G_k$. Before doing this, note that since $\gamma(x)\in \SV$, then $\frac{1}{\gamma(x)}$, $\frac{\ln(x)}{\gamma(x)}$ and $\frac{\ln^2(x)}{\gamma^2(x)}$ are also slowly varying. Thus, it is possible to find a $n_1=n_1(\beta,\alpha)$ such that for every $n>n_1$:
    \begin{align}
        \frac{1}{2n^{\frac{\beta-1}{2}}} \left(\frac{1}{\gamma(n)}+ \beta\frac{\log{n}}{\gamma(n)}\right)<1,\label{eq:notgrowingfastthanpoly}\\
        n^{-\alpha/2}\beta^2\frac{\ln^2{n}}{\gamma^2(n)}<1, \label{eq:unifdsboundintheend}\\
        \frac{\sum_{k=1}^n\left(\frac{\ln{k}}{\gamma(k)}\right)^2 }{n\left(\frac{\log{n}}{\gamma(n)}\right)^2}\leq 2. \label{eq:Karamathaintegralcontrol} 
    \end{align}
    
    Computing the expectation of $G_k$, one gets:
    \begin{align*}
        \Ex{G_k}&= 1+ \sum_{k=1}^{\beta\gamma_k^{-1}\ln{k}}\left[(1-\gamma_k)^k- (1-\gamma_k)^{\beta\gamma_k^{-1}\ln{k}}\right]\\
        &= \frac{1}{\gamma_k}- (1-\gamma_k)^{\beta\gamma_k^{-1}\ln{k}}\left(\frac{1}{\gamma_k}+\beta\gamma_k^{-1}\ln{k}\right)\\
        &\leq \Ex{\G{\gamma_k}}- \frac{1}{2k^{\beta}}\left(\frac{1}{\gamma_k}+ \beta\frac{\ln{k}}{\gamma_k}\right).
    \end{align*} So, by equation \eqref{eq:notgrowingfastthanpoly}, since $\beta>1$ and $\beta-\frac{\beta-1}{2}>1$, one finds that $\Ex{\G{\gamma_k}-G_k}$ is summable. So, there exists $C=C(\beta)<\infty$, such that: 
    \begin{align*}
        \lim_{n\to \infty} \Ex{S_n(\gamma)}-\Ex{\sum_{k=1}^n G_k}\leq C<\infty.
    \end{align*} Since $C$ is a finite constant, there exists $n_2'=n_2'(\beta,\alpha)$ such that for every $n>n_2'$, one gets that $C<\frac{1}{2}\sqrt{n^{\alpha}}$.

    Now, using MCDiarmid's inequality for the truncated random variable $G_k$, for any $\alpha>1$, one gets that: 
    \begin{align*}
        \P{\left|\sum_{k=1}^n \left(G_k - \Ex{G_k}\right)\right|>\sqrt{n^{\alpha}}}\leq 2\exp\left\{-\frac{2n^{\alpha}}{\beta^2\sum_{k=1}^n (\gamma_k^{-1}\ln{k})^2}\right\}
    \end{align*} Moreover, by equation \eqref{eq:Karamathaintegralcontrol} and \eqref{eq:unifdsboundintheend} , for $n>n_1$:
    \begin{align*}
        \P{\left|\sum_{k=1}^n \left(G_k - \Ex{G_k}\right)\right|>\sqrt{n^{\alpha}}}\leq 2\exp\left\{-\frac{n^{\alpha-1}}{\beta^2 (\gamma_n^{-1}\log{n})^2}\right\}\leq 2 e^{-n^{\frac{\alpha-1}{2}}}.
    \end{align*}Which is summable. Thus, for every $\e>0$, there exists $n_2=n_2(\e,\alpha,\beta)$ such that:
    \begin{align*}
        \P{\left|\sum_{k=1}^n \left(G_k - \Ex{G_k}\right)\right|>\sqrt{n^{\alpha}},\forall n>n_2}\leq\frac{\e}{2}.
    \end{align*}
    
    Now, taking $n_0=\max\{n_1,n_2,n_2'\}$, by the triangular inequality, it follows that:
    \begin{align*}
         \mathbb{P}\left(\left|S_n(\gamma)-\sum_{k=1}^n \frac{1}{\gamma_k}\right|>\frac{\sqrt{n^{\alpha}}}{2} , \,\forall n>n_0\right) &\leq \frac{\e}{2}+  \P{\left|\sum_{k=1}^n G_k-\sum_{k=1}^n \frac{1}{\gamma_k}\right|>\frac{\sqrt{n^{\alpha}}}{2},\forall n>n_0}\\
         \leq\frac{\e}{2}+  &\P{\left|\sum_{k=1}^n( G_k-\Ex{G_k})\right|>\sqrt{n^{\alpha}},\forall n>n_0}\leq \e.
    \end{align*} As desired.
\end{proof}
\begin{Rmk}
    Assuming Theorem \ref{thm:1} to be true, it is possible to get the truncation random variable and the equation \eqref{eq:Geometricdifferenceboundfinite} to prove that the gap of the objects is bounded by $\ln^2(n)$, as conjectured by Crammer. 
\end{Rmk}

\begin{Prop}\label{Prop:secondconcentrationgood}
   Consider the dependent random variables $P_n$ and $\sum_{k=1}^n \frac{1}{\lambda_k}$, where $\lambda_k$ is equal to $\prod_{j=1}^k\left(1-\frac{1}{P_j}\right)$. For every $\alpha>1$, we have that:
    \begin{align*}
        \P{\left|P_n-\sum_{k=1}^n \frac{1}{\lambda_k}\right|\leq \sqrt{n^{\alpha}},\,\forall n>n_0}\geq 1-\e.
    \end{align*}
\end{Prop}
\begin{proof}
    Fix $\beta=2$, $\alpha>1$, and take any $\e>0$. Following the same steps of Proposition \ref{Prop:highprobkaramatharate}, it is possible to conclude that exists $x_0=x_0(\e,\alpha)$ such that:
    \begin{align}
        \P{\frac{1}{2x^{\frac{\beta-1}{2}}} \left(\frac{1}{\lambda(x)}+ \beta\frac{\ln{x}}{\lambda(x)}\right)<1,\,\forall x>x_0}>1-\frac{\e}{3}. \label{eq:Pnotgrowingfastthanpoly}\\
        \P{ x^{-\frac{\alpha-1}{2}}\beta^2\frac{\ln^2{x}}{\lambda^2(x)}<1, \forall x>x_0}>1-\frac{\e}{3}\label{eq:Pnlognotgworingfast} \\
        \P{\frac{\int\limits_{1}^x\left(\frac{\ln{y}}{\lambda(y)}\right)^2 dy}{x\left(\frac{\ln{x}}{\lambda(x)}\right)^2}\leq 2,\,\forall x>x_0}>1-\frac{\e}{3}. \label{eq:PKaramathaintegralcontrol} 
    \end{align}
    So, analogous to equations \eqref{eq:notgrowingfastthanpoly}, \eqref{eq:unifdsboundintheend}, and \eqref{eq:Karamathaintegralcontrol}, every sequence $\gamma$ that satisfies the events in equations \eqref{eq:Pnotgrowingfastthanpoly}, \eqref{eq:Pnlognotgworingfast} and \eqref{eq:PKaramathaintegralcontrol}, has its difference between the expectation of $S_n(\gamma)$ and $\sum_{k=1}^nG_k$ bounded by $C=C(\beta)$. In addition, the probability of the event $\{\sum_{k=1}^n G_k -\Ex{G_k}|>\sqrt{n^\alpha}\}$ is uniformly bounded by $e^{-n^{\frac{\alpha-1}{2}}}$; therefore, it is possible to find $n_1=n_1(\alpha,\e)$ such that: 
    \begin{align*}
        \P{\left|P_n-\sum_{k=1}^n \frac{1}{\lambda_k}\right|<\sqrt{n^{\alpha}},\forall n>n_1}<1-\e.
    \end{align*} Finishing the proof of the proposition.
\end{proof}

\section{First term in the  law of large numbers}\label{sec:law1}\noindent

The random sequences $P_n$ and $\sum_{k=1}^n \lambda_k^{-1}$ are both $\mathcal{F}_n$ measurable, which means that knowing the values of $(P_1,\cdots ,P_n)$ results in knowing the value of $\sum_{k=1}^n \lambda_k^{-1}$. In function of that, we are going to shift our random approach into a deterministic study where it will be exposed the properties of deterministic sequences $(P_n)_n$ that satisfy the event $\left\{\left|P_n-\sum_{k=1}^n \frac{1}{\lambda_k}\right|<\sqrt{n^{\alpha}},\forall n>n_0\right\}$. The easiest example of that affirmation is a direct consequence of the fact that $\lambda_n\in \SV$ with probability one, that is: 
\begin{Prop}\label{Prop:concentrationfirt}
     For every $\e>0$ and $\delta>0$, there exists $n_0=n_0(\e,\delta)$ such that:  
    \begin{align*}
        &\P{P_n\leq n^{1+\delta},\, \forall n>n_0}\geq1-\e, \text{ and}\\
        &\P{(1+\delta)n\lambda_n^{-1}\geq P_n\geq (1-\delta)n\lambda_n^{-1},\,\forall n>n_0}\geq 1-\e.
    \end{align*}
\end{Prop}
\begin{proof}
     This proposition is a direct consequence of the results in Propositions \ref{Prop:secondconcentrationgood} and \ref{Prop:highprobkaramatharate}, where by knowing that $\lambda_n$ is slowly varing and its sum is close to the value of $P_n$, one can give bound to the value of $P_n$. For every $\alpha\in (1,2)$, there exists $n_0$ such that with high probability the event $\{|P_n-\sum_{k=1}^n \lambda_k^{-1}|<\sqrt{n^{\alpha}},\,\forall n>n_0\}$ occurs. Moreover, since $\sum_{k=1}^{n}\lambda_k^{-1}$ is close to $n\lambda_n^{-1}$, and slowly varying functions do not decay faster than the polynomial, the proof is over. 
\end{proof}

Before continuing, notice that the definition of the random parameter $\lambda_n$ in equation \eqref{eq:Definiçãolambdan}, was only used to proof Lemma \ref{Lem:lambdaslowv} which proves that the sequence $(\lambda_n)_n$ is slowly varying.   All the results presented before are still valid assuming that the sequence $(\gamma_n)_n$ satisfy the properties of the Lemma \ref{Lem:lambdaslowv}. Now, to show that the sequence $P_n$ approaches the value of $n\log{n}$, the formula of the expression \eqref{eq:Definiçãolambdan} will be used. 

The choice of the value $n\log{n}$ is not random. It comes within a study of a continuous differential equation derived from Proposition \ref{Prop:concentrationfirt}. To briefly explain this choice as a limit, take $\sum_{k=1}^n \prod_{j=1}^k \left(1-\frac{1}{P_j}\right)^{-1}$ and approximate $1-x$ by $e^{-x}$. Then, exchanging summations with integrals, our concentration is related with solutions $f(x)$ of the following integral inequality:
 \begin{align*}
     \left|f(x)-\int\limits_0^x dy \exp\left\{\int\limits_0^y \frac{dz}{f(z)}\right\}\right|<\sqrt{x^{\alpha}}. \,\forall x>x_0. 
 \end{align*} In particular, analyzing functions  $f(x)$ such that $f(x)=\int\limits_0^x dy \exp\left\{\int\limits_0^y \frac{dz}{f(z)}\right\}$, the differential equation $f'(x)=\ln{f(x)}$ appears; Such differential equation have as solution the inverse function $\lii(x)$, where for every $x\in \R\setminus \{1\}$, define the logarithmic integral function:
 \begin{align*}
     \li(x)=\int\limits_{0}^x \frac{dt}{\ln{t}}.
 \end{align*} Where at $t=1$, for every $x>1$, the integral is interpreted as a  a principal value; See more of the function in \cite{Abramowitz1965-my}. Moreover, analyzing all such solutions for any possible initial value, one gets that all solutions are represented by $\frac{1}{c} \lii\left({cx}\right)$ for some $c>0$, that imply in a leader term of order $n\log{n}+n\log{\log{n}}$. The third term can be different in this set of solutions, and is the reason why we expect Conjecture \ref{Conj:1} to be true.  

With a limit candidate, it will be proved that the random variable $P_n$ fluctuates around $n\log{n}$.

 \begin{Lem}\label{Lem:Pnmaiorquenlnne}
     For every $\e>0$, and integer $N>0$ we have that:
     \begin{align*}
         \P{P_n> (1+\e)n\log{n},\forall n>N}=\P{P_n<(1-\e)n\log{n}, \,\forall n>N}=0. 
     \end{align*} 
 \end{Lem}
 \begin{proof} 
 This proof consists in showing that bounding the value of $P_n$ creates a bound for the summation $\sum_{k=1}^n \lambda_n^{-1}$ that do not grows in the right scale. Without loss of generality, fix $\delta>0$ and  assume $N$ large enough such that for every $n>N$, one get:
 \begin{align*}
     1+\delta \geq \frac{\int_1^n \ln^{\frac{1}{1+\e}}(t)dt}{n\ln^{\frac{1}{1+\e}}(n)}\geq 1-\delta\\
     1+\delta \geq \frac{\int_1^n \ln^{\frac{2}{2-\e}}(t)dt}{n\ln^{\frac{2}{2-\e}}(n)}\geq 1-\delta.
 \end{align*} Where it is used Karamatha's results, Proposition \ref{Prop:Karamathalimit1}, since $\ln^{\beta}(x)$ is slowly varying for every $\beta\geq 0$.
 
    The argument explored here, start by observing that $P_n$ and the summation $\sum_{k=1}^n \lambda_k^{-1}$ are close together with high probability. The next step is to find a special region, where if the generator $P_n$ belongs to it, the distance of $P_n$ to $\sum_{k=1}^n \lambda_k^{-1}$ increases faster than $\sqrt{n^{\alpha}}$. Implying that the generator with high probability does not belong to such region. 
    
     Our region is the set of sequences where $\{P_n>(1+\e)n\log{n},\forall n>N\}$ occurs. Start by dividing the summation in the time $N$, so one get:
     \begin{align}\label{eq:sumlambdameantilln}
         \sum_{k=1}^{n} \lambda_{k}^{-1}&= \sum_{k=1}^{N} \lambda_{k}^{-1} + \lambda_N^{-1} \sum_{k=N}^n \prod_{j=N+1}^n \left(1-\frac{1}{P_{j}}\right)^{-1}.
     \end{align} Now,  assuming that $\{P_n>(1+\e)n\log{n},\,\forall n>N\}$ happens, it is possible to bound:  
     \begin{align}\label{eq:importanteboudnup}
         \sum_{k=N}^n \prod_{j=N+1}^k \left(1-\frac{1}{P_{j}}\right)^{-1}&\leq \sum_{k=N}^n \exp\left\{\sum_{j=N+1}^k\frac{1}{P_{j}}\right\}\nonumber\\
         &\leq \int\limits_N^{n} \exp\left\{\int\limits_N^t \frac{dx}{\left(1+\e\right)x\ln{x}}\right\} dt=\int\limits_{N}^n \frac{\ln^{\frac{1}{1+\e}}(t)}{\ln^{\frac{1}{1+\e}}(N)} dt\nonumber\\
         &\leq \frac{(1+\delta)n\ln^{1/(1+\e)}(n) - (1-\delta) N\ln^{1/(1+\e)}(N)}{\ln^{\frac{1}{1+\e}}(N)}.
     \end{align} Since $N$ is a fixed value, taking $n$ to infinity in equation \eqref{eq:sumlambdameantilln}, the growth of $\sum_{k=1}^{n} \lambda_{k}^{-1}$ is bounded above by some constant (depending in $N$) times  $n\ln^{\frac{1}{1+\e}}(n)$. So, eventually, this is smaller than $(1+\e)n\log{n}$. Since with high probability the values of $(P_n)_n$ follow up the summation $\sum_{k=1}^{n} \lambda_{k}^{-1}$, the generators cannot stay forever above $(1+\e)n\log{n}$. 

     In the other case, when $\{P_n<(1-\e)n\log{n},\forall n>N\}$ occurs, it is also possible to assume, without loss of generality, that $N$ is large enough so that $x< \frac{1}{(1-\e)N\log{n}}$,  in particular  $1-(1-\e)x\geq e^{\left(1-\frac{\e}{2}\right)x}$, and:
     \begin{align}\label{eq:(1-e)bound}
          \sum_{k=1}^{n} \lambda_{k}^{-1}&\geq \sum_{k=N}^n \prod_{j=N+1}^k \left(1-\frac{1}{P_{j}}\right)^{-1}\\
          &\geq \frac{(1+\delta)n\ln^{1/\left(1-\frac{\e}{2}\right)}(n) - (1-\delta) N\ln^{1/\left(1-\frac{\e}{2}\right)}(N)}{\ln^{\frac{1}{1-\frac{\e}{2}}}(N)}.\nonumber
     \end{align} Again, notice that $n\ln^{1/(1-\frac{\e}{2})}(n)$ grows faster than $n\log{n}$, and the proof is over. 
 \end{proof}

\begin{Rmk}
    Given $c_1>c_2>0$ and some $a>0$. Consider the functions $c_1n\ln^{\frac{1}{1+a}}{n}$ and $c_2n\ln{(n)}$. Both at $n=1$ are zero,  but since $c_1>c_2$, initially $c_1n\ln^{\frac{1}{1+a}}{n}\geq c_2n\log{n}$. Then, when $n>\exp\left\{\left(\frac{c_1}{c_2}\right)^{\frac{1+a}{a}}\right\}$, these functions invert and one gets $c_1n\ln^{\frac{1}{1+a}}{n}\leq c_2n\log{n}$. In the renewal covering of the natural numbers, $c_1$ grows when $n$ is large and the constant $a$ is usually small; thus, all results treated in this paper take an extremely long time to occur. 
\end{Rmk}

The above Lemma with some careful treatment can give the law of large numbers of the random variables $P_n$, by providing that these different orders of growth blocks the process to recover.

Now, it is time to proof Theorem \ref{thm:1}. 

\begin{proof}[Proof of Theorem \ref{thm:1}]
The proof is divided into two analogous parts: The first part shows that $\limsup_n \frac{P_n}{n\log{n}}=1$ almost surely, and the second part shows that $\liminf_n \frac{P_n}{n\log{n}}=1$ with probability one. Since the demonstrations are analogous, it will be proven just that $\limsup_n \frac{P_n}{n\log{n}}=1$ almost surely. The idea of the proof consists of dividing the space into layers, like an onion, where the layers are of the form $(1+ka)n\log{n}$, for any fixed $a>0$ and $k\in \N $. When $n$ is large, if the process tries to escape the value $n\log{n}$ and enters one of these onion layers, one can use the bound founded in Lemma \ref{Lem:Pnmaiorquenlnne} to show that the summation $\sum_{k=1}^n \lambda_k^{-1}$ never reach the next layer. Therefore, the process that oscillates around $n\log{n}$, and is close to the sum $\sum_{k=1}^n \lambda_k^{-1}$, will not be able to grow larger than $(1+a)n\ln{n}$, so it concentrates bellow $n\ln{n}$. Moreover, the same can be done in the layers $(1-ka)n\log{n}$ to prove that $\liminf_{n}\frac{P_n}{n\log{n}}=1$ almost surely. Finishing the proof of Theorem \ref{thm:1}.

Let any $\e>0$, $a>0$ and $\alpha\in(1,2)$, choose $\delta>0$ small enough so that $(1+a)\frac{1+\delta}{1-\delta}<1+2a$. Then, find $n_0=n_0(\e,\delta)$ such that with probability $1-\e$, for every $n>n_0$, the sequences $(P_n)_n$, $(\lambda_n)_n$  and the function $\ln^{\frac{1}{1+a}}(n)$ satisfies: 
\begin{align}\label{eq:conditions1boa}
    \left|P_n-\sum_{k=1}^n \lambda_k^{-1}\right|<\sqrt{n^{\alpha}}\\
    1+\delta\geq \frac{\sum_{k=1}^n \lambda_k^{-1}}{n\lambda_n^{-1}}\geq 1-\delta, \label{eq:conditions2boa}\\
    1+\delta\geq  \frac{\int\limits_1^n \ln^{\frac{1}{1+a}}(x) dx}{n\ln^{\frac{1}{1+a}}(n)}\geq 1-\delta. \label{eq:conditions3boa}
\end{align}

When the process is around $n\log{n}$ and tries to reach the values of order $(1+2a)n\log{n}$, it has to reach the value $(1+a)n\log{n}$. Moreover, there exists an up-crossing event: a moment where the process is smaller or equal to $(1+a)n\log{n}$, and eventually reaches $(1+2a)n\log{n}$ without being smaller then $(1+a)n\log{n}$. The bounds proposed works in the upcrossing event, and shows that the process cannot reach the next layer.  See in Figure \ref{fig:1} a representation of what is going to happen in the proof. 

\begin{figure}[h!]
    \centering
    \begin{tikzpicture}
        \draw (-2,0)--(7,0) node[anchor=west] {$x\ln{x}$};
        \draw (-2,1)--(7,1) node[anchor=west] {$(1+a)x\ln{x}$};
        \draw (-2,2)--(7,2) node[anchor=west] {$(1+2a)x\ln{x}$};
        \draw[red] (-2,0.1)--(-1.8,0.3)--(-1.7,0.4)--(-1.6,0.3)--(-1.5,0.6)--(-1.4,0.7)--(-1.3,0.5)--(-1.2,0.6)--(-1.1,0.9)--(-1,1.1)--(-0.9,0.9)--(-0.8,1.3)--(-0.7,1.6)--(-0.6,1.7)--(-0.5,1.2)--(-0.4,1.3)--(-0.3,1.5)--(-0.2,1.2)--(-0.1,0.9)--(0,1);
        \filldraw[red] (0,1) circle (2pt);
        \draw[thick] (0,1) ..controls (0.1, 2.5) and (3,1.8).. (5,1);
        \draw[red] (0,1)--(0.2,1.3)--(0.3,1.3)--(0.4,1.5)--(0.5,1.3)--(0.6,1.6)--(0.8,1.8)--(1,1.3)--(1.1,1.5)--(1.3,1.6)--(1.6,1.1)--(1.7,1.3)--(2,1.7);
        \filldraw[red] (2,1.7) circle (1pt);
    \end{tikzpicture}
    \caption{With the space divided in layers of the form $(1+ka)x\ln{x}$, the process in red is trying to cross the layer $(1+a)n\log{n}$, but it is being blocked by a black bound}
    \label{fig:1}
\end{figure}
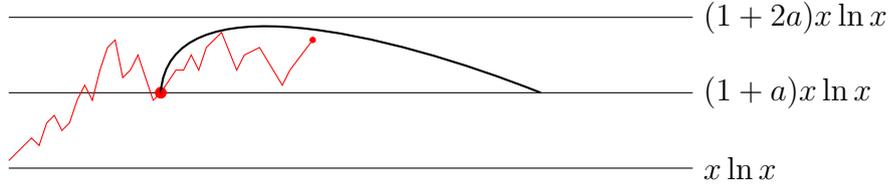

With $a>0$ fixed and $n>n_0$, to simplify some calculations, assume the following. First, suppose that $\sum_{k=1}^n \lambda_k^{-1}=(1+a)n\log{n}$. The second assumption is to fix $(\lambda_n)_n$ and $(P_n)_n$ that satisfies conditions \eqref{eq:conditions1boa} and \eqref{eq:conditions2boa}, that is always possible with high probability. Then, the proof will reach an absurdity when assuming the following condition that imitates the upcrossing event, that is: for every $m>n$, until $P_m>(1+2a)m\ln{m}$, the process satisfies $P_m>(1+a)m\ln{m}$.

Consider $\beta>1$, computing the summation $\sum_{k=1}^{n^{\beta}} \lambda_{k}^{-1}$ in the interval $[1,n^{\beta}]$, where in the value $n$, the process satisfied that $\sum_{k=1}^n \lambda_k^{-1}=(1+a)n\log{n}$, and by hypotheses \eqref{eq:conditions2boa} it is true that $\lambda_n^{-1}\leq \frac{(1+a)\ln{n}}{(1-\delta)}$. Then, it is possible to get using inequality \eqref{eq:importanteboudnup} the following:
     \begin{align}
         \sum_{k=1}^{n^{\beta}} \lambda_{k}^{-1}&= \sum_{k=1}^{n} \lambda_{k}^{-1} + \lambda_n^{-1} \sum_{k=n+1}^{n^{\beta}} \prod_{j=n+1}^k \left(1-\frac{1}{P_{j}}\right)\nonumber\\
         &\leq (1+a)n\log{n}+ (1+a)\frac{\log{n}}{1-\delta}\frac{\left((1+\delta)n^{\beta}\ln^{\frac{1}{1+a}}(n^{\beta}) - (1-\delta) n\ln^{\frac{1}{1+a}}(n)\right)}{\ln^{\frac{1}{1+a}}(n)}\nonumber \\
         &= \frac{(1+a)(1+\delta)}{1-\delta}\frac{\beta^{1/(1+a)}}{\beta} n^{\beta}\ln{n^{\beta}}\leq \frac{(1+a)(1+\delta)}{1-\delta} n^{\beta}\ln{n^{\beta}}.\label{eq:boundintheonion}
     \end{align} By the choice of $\delta>0$, for any $\beta>1$ the sum value never reaches the curve $(1+2a)x\ln{x}$, as desired. 

   To finish the proof, let us generalize our initial assumptions. The first assumption is that the function starts at the value $(1+a)n\log{n}$, to generalize this affirmation, consider the process between $(1+a)n\log{n}$ and $\left(1+\frac{3a}{2}\right)n\log{n}$.Therefore, analyzing the process starting from the first point where an upcrossing $(1+a)n\log{n}$ to $\left(1+\frac{5}{2}a\right)n\log{n}$ occurs, the error in the initial position of the crossing will not interfere with the bound. 
   
    Moreover, to finish the proof, since for every $a>0$, by Lemma \ref{Lem:Pnmaiorquenlnne}, one sees that $P_n$ cannot stay greater than $(1+a)n\log{n}$ for ever, so it must be close to $n\log{n}$. Thus, eventually, for every $a>0$, when $n$ is large, it will return and will never do an upcrossing to $(1+2a)n\log{n}$ again. Since $a>0$ is arbitrary, this get that: 
     \begin{align*}
        \P{\limsup_{n\to \infty} \frac{P_n}{n\log{n}}=1}=1. 
     \end{align*}

     The proof of $\liminf_{n\to \infty}\frac{P_n}{n\log{n}}=1$ is analogous, just taking the $n$ large enough to get the inequality \eqref{eq:(1-e)bound}.  
\end{proof}

\bibliographystyle{plain}
\bibliography{ref}

\end{document}